\documentclass[10pt]{article}
\usepackage{amsfonts}
\usepackage{amsmath}
\usepackage{amssymb}
\usepackage{color}
\usepackage{cite}

\usepackage{graphicx}
\usepackage{amsmath,amsfonts,amssymb,amsthm,epsfig,epstopdf,titling,url,array}
\providecommand{\U}[1]{\protect\rule{.1in}{.1in}}

\theoremstyle{plain}
\newtheorem{thm}{Theorem}[section]
\newtheorem{lem}{Lemma}[section]
\newtheorem{prop}{Proposition}[section]
\newtheorem{definition}[lem]{Definition}

\newtheorem{rem}{Remark}[section]

\newcommand{\pr}{\prime}

\newcommand{\R}{\mathbb{R}}

\newcommand{\E}{\mathbb{E}}
\newcommand{\PP}{\mathbb{P}}

\newcommand{\F}{\mathcal{F}}
\newcommand{\lt}{\left}
\newcommand{\rt}{\right}

\begin{document}

\title{Mean-field optimal control problem of SDDEs driven by fractional Brownian motion}
\author{ Soukaina Douissi \thanks{ Laboratory LIBMA, Faculty
Semlalia, University Cadi Ayyad, Marrakech, Morocco. Email: \texttt{douissi.soukaina@gmail.com}.}, Astrid Hilbert \thanks{Department of Mathematics, Linnaeus University, Sweden. Email: \texttt{astrid.hilbert@lnu.se}.}, Nacira Agram \thanks{Department of Mathematics, University of Oslo, P.O. Box
1053 Blindern, N--0316 Oslo, Norway. Email: \texttt{naciraa@math.uio.no.}}}

\maketitle

\begin{abstract}
We consider a mean-field optimal control problem for stochastic differential equations with delay driven by fractional Brownian motion with Hurst parameter greater than one half. Stochastic optimal control problems driven by fractional Brownian motion can not be studied using classical methods, because the fractional Brownian motion is neither a Markov pocess nor a semi-martingale. However, using the fractional White noise calculus combined with some special tools related to the differentiation for functions of measures, we establish and proove necessary and sufficient stochastic maximum principles. To illustrate our study, we consider two applications: we solve a problem of optimal consumption from a cash flow  with delay and a linear-quadratique (LQ) problem with delay.
\end{abstract}

\vspace{0.5cm}

\section{Introduction}
The interest for stochastic delayed differential equations is constantly increasing. They are frequently used to model the evolution of systems with past-dependence nature. Such systems usually appear in biology, engineering and mathematical finance.\\
There is a rich litterature on stochastic optimal problems with delay. A lot of authors studied both the case where the stochastic systems are driven by a classical Brownian motion as well as where there is jumps, see,e.g., \cite{OS,OSZ,CW,GM}.
\\
Stochastic control problems driven by fractional Brownian motion (fBm) were also studied by many authors, see,e.g., \cite{HOB,BLPR,HZ,WCH}. However, compared with the papers on stochastic control problems driven by the classical Brownian motion, few has been done because classical methods to solve control problems can not be used direclty, since the fractional Brownian motion is not a semi-martingale and not a Markov process. \\
Mean-field problems have also attracted wide attention recently, due to their several applications in physics, economics, finance and stochastic games. Mean-field games were first studied by Lasry and Lions \cite{LasryLions}. Buckdhan, Li and Peng \cite{BLP} studied a special mean-field games and introduced the so-called mean-field backward stochastic differential equations. Later, Carmona and Delarue \cite{carmona2} studied a class of mean-field forward-backward stochastic differential equations and gave many applications.\\
In this paper, all the previous fields are combined to study the optimal control problem of mean-field stochastic differential equations driven by fractional Brownian motion. The dynamic of the controlled state process depends on the state, the control, their laws but also on their values at previous time instants. \\
The dynamics of this work are close to those in the paper of Qiuxi Wang, Feng Chen and Fushan Huang \cite{WCH}. In \cite{WCH}, the adjoint equation is an anticipated backward stochastic diferential equation (ABSDE) driven by both a fBm and a standard Brownian motion and the integral with respect to the fBm, is defined in the Stratonovich sense. \\
This work is also inspired by the recent paper of Buckdhan Rainer and Shuai Jing \cite{BLPR}, if the system has a past-dependence feature. In \cite{BLPR} the dynamic of the adjoint process is driven by a standard Brownian motion, here the anticipated BSDE is driven by a fBM.
\\
The approach used here is an extension of the work of Biagini, Hu, \O ksendal and Sulem \cite{HOB} when the dynamic of the system is of delayed mean-field type. \\
In our paper, we establish and proove necessary and sufficient maximum principles and  we illustrate our study by solving two optimal control problems : a mean-field optimal consumption problem from a cash flow with delay and a linear-quadratique (LQ) problem with delay.\\
We present now more specifically the general problem we consider:

\section{Statement of the problem}
Let $B^{H}$ be a fractional Brownian motion on a filtred probability space $(\Omega, \mathcal{F}, \mathbb{F} = (\mathcal{F}_{t})_{t \geq 0}, \mathbb{P})$. \\
We consider a mean-field controlled stochastic delay equation of the form:
\begin{equation*}
\begin{array}
[c]{lll}%
dX(t) & = & b(t,X(t),X(t-\delta),M(t),M(t-\delta),u(t))dt +\sigma(t,M(t),M(t-\delta))dB^{H}_{t},\text{ }t\in\left[  0,T\right]
,\text{\ }\\
X(t) & = & x_{0}(t);\text{ \ }t\in\left[  -\delta,0\right]  ,
\end{array}  \label{eq:1}%
\end{equation*}
where 

\begin{equation*}
M(t):=\mathbb{P}_{X(t)}, \text{ \ } M(t-\delta):=\mathbb{P}_{X(t-\delta)}
\end{equation*}  

and $T > 0$, $\delta > 0$ are given constants. Here
\vspace{0.2cm}
\begin{equation*}
b: \Omega \times [0,T] \times \mathbb{R} \times \mathbb{R} \times \mathcal{P}_{2}(\mathbb{R})\times\mathcal{P}_{2}(\mathbb{R}) \times \mathcal{U} \rightarrow \mathbb{R},
\end{equation*}
\begin{equation*}
\sigma : \left[  0,T\right]  \times \mathcal{P}_{2}(\mathbb{R})\times\mathcal{P}_{2}(\mathbb{R}) \rightarrow \mathbb{R},
\end{equation*}

\vspace{0.2cm}
are given functions such that, for all $t \in [0,T]$, $b(.,t,x,\bar{x},m,\bar{m},u)$ is supposed to be $\mathcal{F}_{t}$- measurable for all $x,\bar{x} \in \mathbb{R}$, $ u \in \mathcal{U}$, $m, \bar{m} \in \mathcal{P}_{2}(\mathbb{R})$. The function $\sigma$ is assumed to be deterministic such that its integral with respect to the fBm will be a Wiener type integral. $\mathcal{P}_{2}(\mathbb{R})$ denotes the space of all probability measures $m$ on $(\mathbb{R}, \mathcal{B}(\mathbb{R}))$, such that $\int_{\mathbb{R}} |x|^{2}m(dx) < \infty$. \\
\\
The function $x_{0}$ is assumed to be continuous and deterministic. The set $\mathcal{U}\subset%
\mathbb{R}
$ consists of the admissible control values. The information available to the
controller is given by the filtration $\mathbb{F}$ generated by the  fBm $B^{H}$. The set of admissible controls denoted by $\mathcal{A}_{\mathbb{F}}$ are the strategies available
to the controller, required to be  $\mathcal{U}$-valued and $\mathbb{F}$-adapted processes. Through the paper, we assume  that $X$ exists and belongs to $L^{2}(\Omega \times [0,T])$.
\\
\\
The performance functional is assumed to have the form 
\begin{equation*}%
\begin{array}
[c]{lll}%
J(u) & = & \mathbb{E[}g(X(T),M(T))+%
{\textstyle\int_{0}^{T}}
f(t,X(t),X(t-\delta),M(t),M(t-\delta),u(t))dt]
\end{array}
, \label{P}%
\end{equation*}
where $f:\Omega\times\left[
0,T\right]  \times%
\mathbb{R}
\times \mathbb{R} \times \mathcal{P}_{2}(\mathbb{R})\times\mathcal{P}_{2}(\mathbb{R}) \times\mathcal{U}\rightarrow%
\mathbb{R}
$ and $g:\Omega  \times%
\mathbb{R}
\times\mathcal{P}_{2}(\mathbb{R})\rightarrow%
\mathbb{R}
$ are given functions, such that for all $ t \in [0,T]$, $f(.,t,x,\bar{x},m,\bar{m},u)$ is assumed to be $\mathcal{F}_{t}$-measurable for all $x,\bar{x} \in \mathbb{R}$, $u \in \mathcal{U}$,  $m,\bar{m} \in \mathcal{P}_{2}(\mathbb{R})$. The function $g(.,x,m)$ is assumed to be $\mathcal{F}_{T}$-measurable for all $x \in \mathbb{R}$, $m \in \mathcal{P}_{2}(\mathbb{R})$.

\vspace{0.3cm}
We also assume the following integrability condition
\begin{equation*}%
\begin{array}
[c]{lll}%
 \mathbb{E[}|g(X(T),M(T))|+%
{\textstyle\int_{0}^{T}}
|f(t,X(t),X(t-\delta),M(t),M(t-\delta),u(t))|dt] < +\infty.
\end{array}
, \label{PP}%
\end{equation*}
The functions $\sigma$, $b$, $f$ and $g$ are assumed to be continuously differentiable with respect to $x, \bar{x},u$ with bounded
derivatives and admit Fr\'{e}chet bounded derivatives with respect to $m, \bar{m}$.\\

The problem we consider in this paper  is the following :\\
\\
\textbf{Problem:} Find a control $u^{*} \in \mathcal{A}_{\mathbb{F}}$ such that 
\begin{equation}
\label{eq:optimal perf}
J(u^*)=\sup_{u\in \mathcal{A}_{\mathbb{F}}}J(u).
\end{equation}
any control $u^{*} \in \mathcal{A}_{\mathbb{F}}$ satisfying (\ref{eq:optimal perf}) is called an optimal control. 
\section{Generalities}
In this section we give some preliminaries concerning fractional Brownian motion based on fractional White noise calculus and some generalities on differentiability with respect to the measures. For a general introduction to fractional White noise theory the reader may consult the books \cite{BHOZ,HXO}. 

\subsection{Fractional Brownian motion}

Let $T > 0$  be a finite time horizon and $\Omega$ be the space $\mathcal{S}^{{\prime}}([0,T])$
of tempered distributions on $[0,T]$, which is the dual of the Schwartz
space $\mathcal{S}([0,T])$ of rapidly decreasing smooth functions on
$[0,T]$. \\
\\
For $  1/2 < H < 1$. We put 
\begin{equation*}
\varphi_{H}(t,s) = H(2H-1)|t-s|^{2H-2}, \hspace{0.5cm} s,t \in [0,T].
\end{equation*}
\\
If $\omega\in\mathcal{S}^{{\prime}}([0,T])$ and
$f\in\mathcal{S}([0,T])$ , we let $<\omega,f>=\omega(f)$ denote the
action of $\omega$ applied to $f$, it can be extended to $f: [0,T]%
\rightarrow\mathbb{R}$ such that%

\[
\Vert f\Vert_{{H}}^{2}:=%
{\textstyle\int_{0}^{T}}
{\textstyle\int_{0}^{T}}
f(s)f(t)\varphi_{H}(t,s)dsdt<+\infty.
\]

\vspace{0.3cm}
The space of all such functions $f$ is denoted by $\mathbb{L}_{H}%
^{2}([0,T])$. The map $f \mapsto \exp(- 1/2 \Vert f\Vert_{{H}}^{2})$ with $f \in \mathcal{S}([0,T])$, is positive definite on $\mathcal{S}([0,T])$, then by Bochner-Minlos theorem there exists a probability measure $\mathbb{P}^{H}= \mathbb{P}$ on the Borel subsets $\mathcal{B}(\Omega)$ such that 
\begin{equation}\label{BM}
\int_{\Omega} e^{i <\omega,f>} d\mathbb{P}(\omega) = e^{-1/2 \Vert f\Vert_{{H}}^{2}} \hspace{1cm} \forall f \in \mathcal{S}([0,T])
\end{equation}
It follows from (\ref{BM}) that
\begin{equation}\label{3..6}
\mathbb{E}[<.,f>] = 0 \hspace{0.2cm} \text{and} \hspace{0.2cm} \mathbb{E}[<.,f>^{2}] = \Vert f\Vert_{{H}}^{2} 
\end{equation}

where $\mathbb{E}$ denotes the expectation under the probability measure $\mathbb{P}$. Hence, if we put
\begin{equation*}
B^{H}_{t} := <\omega, \chi_{[0,t]}(.)>
\end{equation*}
then using (\ref{3..6}), $B^{H}$ is a fractional Brownian motion with Hurst parameter $H$, that is a centred Gaussian process with covariance function 
\begin{equation*}
\mathbb{E}[B^{H}_{t}B^{H}_{s}] = \frac{1}{2}(t^{2H} + s^{2H} - |t-s|^{2H}), \hspace{0.2cm} s,t \in [0,T]
\end{equation*}
From now on, we endow $\Omega$ with the natural filtration $\mathbb{F}:= \{\mathcal{F}_{t}\}_{t \in [0,T]}$ generated by $B^{H}$.
\begin{lem} If $f,g$ belong to $\mathbb{L}_{H}%
^{2}([0,T])$, then the Wiener integrals $\int_{0}^{T} f_{s}dB^{H}_{s}$ and $\int_{0}^{T} g_{s} dB^{H}_{s}$ are well defined as zero mean Gaussian random variables with variances $\Vert f\Vert_{{H}}^{2}$  and $\Vert g\Vert_{{H}}^{2}$ respectively and 
\begin{equation*}
\mathbb{E}[\int_{0}^{T} f_{s}dB^{H}_{s}\int_{0}^{T} g_{s} dB^{H}_{s}] = \int_{0}^{T} \int_{0}^{T} f(s)g(t)\varphi_{H}(t,s) dt ds
\end{equation*}
\end{lem}
We denote by $\tilde{\mathcal{S}}$ the set of all polynomial functions of $B^H(\psi_j)=\int^T_0\psi_j(t)d B^H(t)$. For an element $F\in \tilde{\mathcal{S}}$, having the form 
$$
F=g(B^H(\psi_1), \cdots, B^H(\psi_n)), 
$$
where $g$ is a polynomial function of $n$ variables, we define its Malliavin derivative $D_s^H F$ by
$$
D^H_s F:=\sum_{i=1}^{n}\frac{\partial g}{\partial x_i}(B^H(\psi_1), \cdots, B^H(\psi_n))\psi_i(s),\qquad  0\le s \le T.
$$

We also introduce another derivative %

\[
\mathbb{D}_{t}^{H}F :=%
{\textstyle\int_{0}^{T}}
D^{H}_{r}F \varphi_{H}(r,t)dr,
\]
\\
Let $\mathcal{L}%
_{H}^{1,2}([0,T])$  be the set of processes $G:[0,T]\times\Omega%
\rightarrow\mathbb{R}$ such that $\mathbb{D}_{s}^{H}G(s)$ exists for all
$s\in [0,T]$ and%

\[
\Vert G\Vert_{\mathcal{L}_{H}^{1,2}}^{2}:=\mathbb{E}[%
{\textstyle\int_{0}^{T}}
{\textstyle\int_{0}^{T}}
G(s)G(t)\varphi_{H}(t,s)dtds+(%
{\textstyle\int_{0}^{T}}
\mathbb{D}_{s}^{H}G(s)ds)^{2}]<\infty.
\]
\\
We let $\int_{0}^{T} G(s)dB^{H}(s)$ denote the fractional Wick-It\^{o}-Skorohod (fWIS) integral of the process $G$ with respect to $B^{H}$. We recall  its construction: if $G$ belongs to the family $\mathbb{S}$ of step functions of
the form%

\[
G(t,\omega)=%
{\textstyle\sum_{i=1}^{N}}
G_{i}(\omega)\chi_{\lbrack t_{i},t_{i+1}[}(t),
\]
where $0\leq t_{1}<t_{2}<...<t_{N+1}\leq T$, then the fWIS integral is defined naturally as follows%

\[%
{\textstyle\int_{0}^{T}}
G(t,\omega)dB^{H}(t):=%
{\textstyle\sum_{i=1}^{N}}
G_{i}(\omega)\diamondsuit(B^{H}(t_{i+1})-B^{H}(t_{i})),
\]
where $\diamondsuit$ denotes the Wick product, see \cite{BHOZ} for its definition. For $G \in\mathbb{S}%
\cap\mathcal{L}_{H}^{1,2}([0,T])$, we have the isometry%
\[
\mathbb{E}[(%
{\textstyle\int_{0}^{T}}
G(t)dB^{H}(t))^{2}]=\mathbb{E}[%
{\textstyle\int_{0}^{T}}
{\textstyle\int_{0}^{T}}
G(s)G(t)\varphi_{H}(t,s)dtds+(%
{\textstyle\int_{0}^{T}}
\mathbb{D}_{s}^{H}G(s)ds)^{2}],
\]
Using this we can extend the integral
$%
{\textstyle\int_{0}^{T}}
G(t)dB^{H}(t)$ to $\mathcal{L}_{H}^{1,2}([0,T])$. Note that if $G_{1}, G_{2}
\in\mathcal{L}_{H}^{1,2}([0,T])$, we have by polarization%

\[
\mathbb{E}[%
{\textstyle\int_{0}^{T}}
G_{1}(t)dB^{H}(t)%
{\textstyle\int_{0}^{T}}
G_{2}(t)dB^{H}(t)]=\mathbb{E}[%
{\textstyle\int_{0}^{T}}
G_{1}(s)G_{2}(t)\varphi_{H}(s,t)dsdt+%
{\textstyle\int_{0}^{T}}{\textstyle\int_{0}^{T}}
\mathbb{D}_{s}^{H}G_{1}(s)\mathbb{D}_{t}^{H}G_{2}(t)dsdt]%
\]
An important property of this integral is that%

\begin{equation*}
\mathbb{E}[%
{\textstyle\int_{0}^{T}}
G(t)dB^{H}(t)]=0; \hspace{0.5cm} \text{ for all} \hspace{0.5cm} G
\in\mathcal{L}_{H}^{1,2}([0,T]). \label{IntegFBM}%
\end{equation*}

\vspace{0.3cm}

We will need the following integration by parts formula.

\begin{prop}
[Integration by parts]\label{integ.part.form} Let $X$ and $Y$ be two processes
of the form
\[
dX(t)=F_{1}(t)dt+ G_{1}(t)dB^{H}(t),\hspace{0.5cm}X(0)=x\in\mathbb{R},
\]
and
\[
dY(t)= F_{2}(t)dt+ G_{2}(t)dB^{H}(t),\hspace{0.5cm}Y(0)=y\in\mathbb{R},
\]

\end{prop}

where $F_{1}:[0,T] \times\Omega\rightarrow\mathbb{R}$, $F_{2}: [0,T]\times\Omega
\rightarrow\mathbb{R}$, $ G_{1}: [0,T]\times \Omega%
\rightarrow\mathbb{R}$ and $ G_{2}: [0,T] \times\Omega \rightarrow
\mathbb{R}$ are given processes such that $G_{1}$, $G_{2}\in\mathcal{L}%
_{H}^{1,2}([0,T])$.

\begin{enumerate}
\item Then, for $T>0$,
\begin{align*}
\mathbb{E}[X(T)Y(T)]  &  =xy+\mathbb{E}[%
{\textstyle\int_{0}^{T}}
X(s)dY(s)]+\mathbb{E}[%
{\textstyle\int_{0}^{T}}
Y(s)dX(s)]\label{IntegPart}\\
&  \hspace*{0.3cm}+\mathbb{E}[%
{\textstyle\int_{0}^{T}}
{\textstyle\int_{0}^{T}}
G_{1}(t)G_{2}(s)\varphi_{H}(t,s)dsdt]+\mathbb{E}[%
{\textstyle\int_{0}^{T}\int_{0}^{T}}
\mathbb{D}_{t}^{H}G_{1}(t)\mathbb{D}_{s}^{H}G_{2}(s)dsdt],
\end{align*}
provided that the first two integrals exist.

\item In particular, if $G_{1}$ or $G_{2}$ is deterministic, then
\[
\mathbb{E}[X(T)Y(T)]=xy+\mathbb{E}[%
{\textstyle\int_{0}^{T}}
X(s)dY(s)]+\mathbb{E}[%
{\textstyle\int_{0}^{T}}
Y(s)dX(s)]+\mathbb{E}[%
{\textstyle\int_{0}^{T}}
{\textstyle\int_{0}^{T}}
G_{1}(t)G_{2}(s)\varphi_{H}(t,s)dsdt].
\]

\end{enumerate}
\subsection{Differentiability of Functions of Measures} 

Let $\mathcal{P}(\R)$ be the space of all probability measures on $(\R, \mathcal{B}(\R))$. We denote by $\mathcal{P}_p(\R)$ the subspace of  $\mathcal{P}(\R)$ of order $p$, which means
\[
\mathcal{P}_p(\R)=\{m\in\mathcal{P}(\R):\int_{\R} |x|^p m(dx)<+\infty\}.
\]
\begin{itemize}
  \item[$\bullet$] \textbf{ The Wasserstein metric :}
  
On $\mathcal{P}_p(\R)$, the Wasserstein metric of order $p$  is defined by
\[
\begin{aligned}
W_p(m,m^{\prime})=\inf\bigg\{\lt(\int_{\R^{2}} |x-y|^p\rho(dx,dy)\rt)^{\frac1p},& \ \rho\in\mathcal{P}_p(\R \times \R)\ \hbox{such that}\\ &\rho(\cdot\times\R)=m \ \hbox{and}\ \rho(\R\times\cdot)= m^{\prime}\bigg\}.
\end{aligned}
\]

Notice that if $\xi$ and $\eta$ are two real $p$-integrable random variables with laws $\PP_\xi$ and $\PP_\eta$, then we have $W_p(\PP_\xi,\PP_\eta)\le\lt(\E[{|\xi-\eta|^p}]\rt)^{\frac1p}$ since we can choose a special $\rho=\PP_{(\xi,\eta)}$ in the above definition. 
\\
\item[$\bullet$] \textbf{ Diffentiability of functions of measures :}
\\
\\
The notion of differentiability for functions of measures that we will use in the paper, is the one introduced by Lions in his course at  Coll\`ege de France \cite{Lions} and summarized by Cardaliaguet \cite{Cardaliaguet}. We also refer to Carmona and Delarue \cite{carmona2}.\\
\\
It's based on the \textit{lifting} of fuctions $m \in \mathcal{P}_2(\R) \mapsto \sigma(m)$ into functions $\tilde{\xi} \in L^{2}(\tilde{\Omega}; \mathbb{R}) \mapsto \tilde{\sigma}(\tilde{\xi})$, over some probability space $(\tilde{\Omega},\tilde{\mathcal{F}},\tilde{\mathbb{P}})$, by setting $\tilde{\sigma}
(\tilde{\xi}) := \sigma(\tilde{\mathbb{P}}_{\tilde{\xi}})$.

\begin{definition}
A function $\sigma$ is said to be differentiable at $ m_{0} \in \mathcal{P}_2(\R)$, if there exists a random variable $\tilde{\xi}_{0}\in L^2(\tilde{\Omega},\tilde{\F},\tilde{\PP})$ over some probability space $(\tilde{\Omega},\tilde{\F},\tilde{\PP})$ with $\tilde{\PP}_{\tilde{\xi}_{0}}= m_{0}$ such that $\tilde{\sigma}: L^2(\tilde{\Omega},\tilde{\F},\tilde{\PP}) \to \R$ is Fr\'echet differentiable at $\tilde{\xi}_{0}$.
\end{definition}

We suppose for simplicity that $\tilde{\sigma}: L^2(\tilde{\Omega},\tilde{\F},\tilde{\PP}) \to \R$ is Fr\'echet differentiable. We denote its Fr\'echet derivative at $\tilde{\xi}_{0}$ by $D\tilde{\sigma}(\tilde{\xi}_{0})$. Recall that $D\tilde{\sigma}(\tilde{\xi}_{0}):L^2(\tilde{\Omega},\tilde{\F},\tilde{\PP}) \to \R$ is a continuous linear mapping; i.e. $D\tilde{\sigma}(\tilde{\xi}_{0})\in L( L^2(\tilde{\Omega},\tilde{\F},\tilde{\PP}),\R)$.
With the identification that $L( L^2(\tilde{\Omega},\tilde{\F},\tilde{\PP}),\R) \equiv L^2(\tilde{\Omega},\tilde{\F},\tilde{\PP})$ given by Riesz representation theorem, $D\tilde{\sigma}(\tilde{\xi}_{0})$ is viewed as an element of 
$L^2(\tilde{\Omega},\tilde{\F},\tilde{\PP})$, hence we can write 
\[
\sigma(m)-\sigma(m_{0})=\tilde{\sigma}(\tilde{\xi})-\tilde{\sigma}(\tilde{\xi}_{0})=
\tilde{\mathbb{E}}[(D\tilde{\sigma})({\tilde{\xi}}_{0}).(\tilde{\xi}-\tilde{\xi}_{0})] +o(\tilde{\mathbb{E}}[|\tilde{\xi}-\tilde{\xi}_{0}|^{2}]^{1/2}),\  \textrm{as}\  \tilde{\mathbb{E}}[|\tilde{\xi}-\tilde{\xi}_{0}|^{2}]^{1/2} \to 0.
\]
where $\tilde{\xi}$ is a random variable with law $m$. Moreover, according to Cardaliaguet \cite{Cardaliaguet}, there exists a Borel function $h_{m_{0}}:\R\to\R$, such that $D\tilde{\sigma}(\tilde{\xi}_{0}) =h_{m_{0}}(\tilde{\xi}_{0})$, $\tilde{\PP}$-a.s. We define the derivative of $\sigma$ with respect to the measure at $m_{0}$ by putting $\partial_m\sigma(m_{0})(x) : =h_{m_{0}}(x)$. Notice that $\partial_m\sigma(m_{0})(x)$ is defined $m_{0}(d x)$-a.e. uniquely. Therefore, the following differentiation formula is invariant by modification of the space $\tilde{\Omega}$ where the random variables $\tilde{\xi}_{0}$ and $\tilde{\xi}$ are defined, i.e.
\[
\sigma(m)-\sigma(m_{0})=\tilde{\E}[
\partial_m\sigma(m_{0})(\tilde{\xi}_{0}).(\tilde{\xi}-\tilde{\xi}_{0})]+o(\tilde{\mathbb{E}}[|\tilde{\xi}-\tilde{\xi}_{0}|^{2}]^{1/2}),\  \textrm{as}\ \tilde{\mathbb{E}}[|\tilde{\xi}-\tilde{\xi}_{0}|^{2}]^{1/2} \to 0. 
\]
whenever $\tilde{\xi}$ and $\tilde{\xi}_{0}$ are random variables with laws $m$ and $m_{0}$ respectively.

\item[$\bullet$] \textbf{ Joint concavity}

We will need the joint concavity of a function on $(\R\times \mathcal{P}_2(\R))$. A differentiable function $b$ defined on  $(\R \times \mathcal{P}_2(\R))$ is concave, if for every   $(x^\pr,m^\pr)$ and $(x,m) \in  (\R \times \mathcal{P}_2(\R))$, we have 
\begin{equation}
\label{concavity}
\begin{aligned}
b(x^{\prime},m^{\prime})-b(x,m)- \partial_x b(x,m) (x^\pr-x) -\tilde{{\E}}[\partial_m b(x,m)(\tilde{X})(\tilde{X}^\pr-\tilde{X})] \leq 0, 
\end{aligned}
\end{equation}
whenever $\tilde{X}, \tilde{X}^\pr\in L^2(\tilde{\Omega},\tilde{\F},\tilde{\PP};\R)$ with laws $m$ and $m^{\prime}$
respectively.
\end{itemize}
\section{Necessary maximum principle}
In this section, we establish a maximum principle of necessary type. \\
\\
For this end, we assume that $\mathcal{U}$ is a closed convex set (and hence $\mathcal{A}_{\mathbb{F}}$ is convex). Now for a given $u^{\ast}\in\mathcal{A}_{\mathbb{F}}$ and an arbitrary but fixed control $u\in\mathcal{A}_{\mathbb{F}}$, we
define
\[%
\begin{array}
[c]{ll}%
u^{\theta}:=u^{\ast}+\theta\left(  u-u^{\ast}\right)  , & \theta\in\left[
0,1\right]  .
\end{array}
\]
Note that, thanks to the convexity of $\mathcal{A}%
_{\mathbb{F}}$,  $u^{\theta}\in\mathcal{A}_{\mathbb{F}}, \text{for all \hspace{0.1cm}} \theta\in\left[
0,1\right]  $. We denote by $X^{\theta}:=X^{u^{\theta}}$ and by $X^{\ast
}:=X^{u^{\ast}}$ the controlled state processes corresponding to $u^{\theta}$ and
$u^{\ast}$\ respectively.
\\
\\
For $u^{\ast}\in\mathcal{A}_{\mathbb{F}}$ and the associated controlled
state process $X^{\ast}$, let $Y(t) := \frac{d}{d \theta} X^{\theta}(t)|_{\theta =0}$, hence $Y$ satisfies the following  SDDE :%
\begin{equation}
\begin{array}
[c]{ll}%
dY(t)& = \{ \partial_{x}b^{*}(t)Y(t)+\partial_{\bar{x}}%
b^{*}(t)Y(t-\delta)+ \mathbb{\tilde{E}}[\partial_{m}b^{*}(t)(\tilde{X}^{*}(t))\tilde{Y}(t)] \\
& \hspace*{0.8cm} + \mathbb{\tilde{E}}[\partial_{\bar{m}}b^{*}(t)(\tilde{X}^{*}(t-\delta))\tilde{Y}(t-\delta)] +\partial_{u}b^{*}(t)(u(t)-u^{\ast}(t))\}dt\\
& \hspace*{1.3cm} +\{\mathbb{\tilde{E}}[\partial_{m}\sigma(t)(\tilde{X}^{*}(t))\tilde{Y}(t)] +  \mathbb{\tilde{E}}[\partial_{\bar{m}}\sigma^{*}(t)(\tilde{X}^{*}(t-\delta))\tilde{Y}(t-\delta)]\}dB_{t}^{H}, \text{ \ } t \in [0,T],
\\
\\
Y(t)  & =0; \hspace{0.5cm} t\in\left[  -\delta,0\right]  \text{.}%
\end{array}
  \label{y}%
\end{equation}

where we used the following notations:

\[%
\begin{array}
[c]{ll}%
\partial_{x}b^{*}(t)& := \partial_{x}b(t,X^{\ast}(t),X^{\ast}(t-\delta),M^{\ast}(t),M^{\ast}(t-\delta),u^{\ast}(t)),\\
\partial_{m}b^{*}(t)(.)& := \partial_{m}b(t,X^{\ast}(t),X^{\ast}(t-\delta),M^{\ast}(t),M^{\ast}(t-\delta),u^{\ast}(t))(.),\\
\partial_{m}\sigma^{*}(t)(.) & := \partial_{m}\sigma(t,M^{\ast}(t),M^{\ast}(t-\delta))(.).
\end{array}
\]
and $(\tilde{X}, \tilde{Y}, \tilde{u})$ is an independant copy of $(X,Y,u)$ defined on some probability space $(\tilde{\Omega}, \tilde{\mathcal{F}}, \tilde{\mathbb{P}})$ and $\tilde{\mathbb{E}}$ denotes the expectation on $(\tilde{\Omega}, \tilde{\mathcal{F}}, \tilde{\mathbb{P}})$. 
\begin{rem}\label{rem}
From the  definition of the tilde random variables and since $\sigma$ is deterministic, we have 
\begin{equation*}
\tilde{\mathbb{E}}[\partial_{m}\sigma^{\ast}(t)(\tilde{X}^{\ast}(t))\tilde{Y}^{\ast}(t)] = \mathbb{E}[\partial_{m}\sigma^{\ast}(t)({X}^{\ast}(t)){Y}^{\ast}(t)].
\end{equation*}
Note that using the previous notations, $\mathbb{\tilde{E}}[\partial_{m}b^{\ast}(t)(\tilde{X}^{\ast}(t))\tilde{Y}(t)]$ is a function of the random vector $(X^{\ast}(t),X^{\ast}({t-\delta}),u^{\ast}(t))$ as it stands for  \\
$\mathbb{\tilde{E}}[\partial_{m}b(t,x,\bar{x},M^{\ast}(t),M^{\ast}(t-\delta),u)(\tilde{X}^{*}(t))\tilde{Y}(t)]|_{x = X^{\ast}(t),\bar{x}=X^{\ast}(t-\delta), u = u^{\ast}(t)}$.
\end{rem}
We assume that the derivative process $Y$ exists and belongs to $L^{2}(\Omega \times [0,T])$ and that the function $\psi_ {\delta}^{*}: t \mapsto \mathbb{\tilde{E}}[\partial_{m}\sigma^{*}(t)(\tilde{X}^{*}(t))\tilde{Y}(t)] +  \mathbb{\tilde{E}}[\partial_{\bar{m}}\sigma^{*}(t)(\tilde{X}^{*}(t-\delta))\tilde{Y}(t-\delta)] $ is  in $\mathbb{L}_{H}^{2}([0,T])$, the integral with respect to the fBm is therefore well defined in the Wiener sense. \\
\\
Now if $u^{\ast}$ is an optimal control, we have $J(u^{\ast})\leq
J(u^{\theta})$, for all $\theta\in\left[  0,1\right]  $, i.e.,%

\begin{equation}
0\leq\underset{\theta\rightarrow0}{\lim}\tfrac{J(u^{\theta})-J(u^{\ast}%
)}{\theta}\text{.} \label{j}%
\end{equation}
with
\begin{align}\label{J}
\underset{\theta\rightarrow0}{\lim}\tfrac{1}{\theta}(J(u^{\theta})-J(u^{\ast
}))
&  = \mathbb{E[}\partial_{x}g^{*}(T)Y(T)+ \mathbb{\tilde{E}}[\partial_{m}g^{*}(T)(\tilde{X}^{*}(T))\tilde{Y}(T)]] \nonumber \\
& +\mathbb{E[}%
{\textstyle\int_{0}^{T}}
\{\partial_{x}f^{*}(t)Y(t)+ \partial_{\bar{x}}f^{*}(t)Y(t-\delta)+ \mathbb{\tilde{E}}[\partial
_{m}f^{*}(t)(\tilde{X}^{*}(t))\tilde{Y}(t)]
  \nonumber \\
  & + \mathbb{\tilde{E}}[\partial
_{\bar{m}}f^{*}(t)(\tilde{X}^{*}(t-\delta))\tilde{Y}(t-\delta)] + \partial_{u}f^{*}(t)\left(
u(t)-u^{\ast}(t)\right)  \}dt].
\end{align}
where we have used the simplified notations :
\[%
\begin{array}
[c]{ll}%
\partial_{x}g^{*}(T)& := \partial_{x}g(X^{\ast}(T),M^{\ast}(T)),\\
\partial_{m}g^{*}(T)(.)& := \partial_{m}g(X^{\ast}(T),M^{\ast}(T))(.),\\
\partial_{x}f^{*}(t)& := \partial_{x}f(t,X^{\ast}(t),X^{\ast}(t-\delta),M^{\ast}(t),M^{\ast}(t-\delta),u^{\ast}(t)),\\
\partial_{{m}}f^{*}(t)(.)& := \partial_{{m}}f(t,X^{\ast}(t),X^{\ast}(t-\delta),M^{\ast}(t),M^{\ast}(t-\delta),u^{\ast}(t))(.),\\
\partial_{m}\sigma^{*}(t)(.) & := \partial_{m}\sigma(t,M^{\ast}(t),M^{\ast}(t-\delta))(.).
\end{array}
\]
In order to determine the adjoint backward equation associated to $(u^{*},X^{*})$, we suppose that it
has in general the following form%

\begin{equation}
\left\{
\begin{array}
[c]{ll}%
dp^{*}(t) & =-\alpha(t)dt+q^{*}(t)dB_{t}^{H}, \text{ \  }t\in\left[  0,T\right]  ,\\
p^{*}(T) & =  \partial_{x} g^{*}(T) +  \mathbb{\tilde{E}}[\partial_{m}\tilde{g}^{*}(T)(X(T))].
\end{array}
\right.  \label{bac}%
\end{equation}

\vspace{0.5cm}
where $(p^{*},q^{*})$ is assumed to be in $\mathcal{L}_{H}^{1,2}([0,T])\times \mathcal{L}_{H}^{1,2}([0,T]) $, the integral with respect to the fBm is a fractional Wick-It\^{o}-Skorohod integral and $\alpha$ is  some $\mathbb{F}$-adapted process  which we have to determine. \\
\\
Applying the integration by parts formula of Proposition \ref{integ.part.form}, to $p^{*}\left(  t\right)$ and $Y\left(  t\right)  $, we obtain

\begin{align*}
\mathbb{E}[p^{*}(T)Y(T)]  &  = \mathbb{E}[\int_{0}^{T}p^{*}(t)dY(t)]+ \mathbb{E}[\int_{0}^{T}Y(t)dp^{*}(t)] + \mathbb{E}[\int_{0}^{T}\int_{0}^{T}
q^{*}(t) \psi_{\delta}^{*}(t)\varphi_{H}(s,t)dsdt] \\
& = \mathbb{E}[\int_{0}^{T}p^{*}(t)\{\partial_{x}b^{*}(t)Y(t)+\partial_{\bar{x}}b^{*}(t)Y(t-\delta
) + \mathbb{\tilde{E}}[\partial_{m}b^{*}(t)(\tilde{X}^{*}(t))\tilde{Y}(t)] \\
& \hspace{0.3cm} + \mathbb{\tilde{E}}[\partial_{\bar{m}}b^{*}(t)(\tilde{X}^{*}(t-\delta))\tilde{Y}(t-\delta)]  + \partial_{u}b^{*}(t) (u(t)-u^{*}(t)) \}dt] - \mathbb{E}[\int_{0}^{T} Y(t) \alpha(t)dt] \\
& \hspace{0.3cm} + \mathbb{E}[\int_{0}^{T}\int_{0}^{T}
q^{*}(s) \psi^{*}_{\delta}(t) \varphi_{H}(s,t)dsdt]
\end{align*}
where we assumed that $Y(t)q^{*}(t) \in \mathcal{L}_{H}^{1,2}([0,T])$ and $\psi^{*}_{\delta}(t)p^{*}(t) \in \mathcal{L}_{H}^{1,2}([0,T])$. By Fubini's theorem, remark (\ref{rem}), remplacing $\psi^{*}_{\delta}(t)$ by its value and by a change of variables using the fact that $Y(t) = 0$ for all $t \in [-\delta,0]$ and  , we get
\begin{align}\label{PY}
\mathbb{E}[p^{*}(T)Y(T)]  & = \mathbb{E}[\int_{0}^{T} Y(t) \{ p^{*}(t) \partial_{x}b^{*}(t) + p^{*}(t+\delta)\partial_{\bar{x}}b^{*}(t+\delta) \chi_{[0,T-\delta]}(t) - \alpha(t)\nonumber \\
&  \hspace{0.6cm} + \mathbb{\tilde{E}}[\tilde{p}^{*}(t)\partial_{m}\tilde{b}^{*}(t)(X^{*}(t))] + \mathbb{\tilde{E}}[\tilde{p}^{*}(t+\delta)\partial_{\bar{m}}\tilde{b}^{*}(t+\delta)(X^{*}(t))]\chi_{[0,T-\delta]}(t) \nonumber \\
& \hspace{0.6cm} + \int_{0}^{T}  \{ \tilde{\mathbb{E}}[\tilde{q}^{*}(s)]\partial_{m}{\sigma}^{*}(t)(X^{*}(t))\varphi_{H}
(t,s)  \nonumber \\
& \hspace{0.6cm}+ \tilde{\mathbb{E}}[\tilde{q}^{*}(s)]\partial_{\bar{m}}{{\sigma}}^{*}(t +\delta)(X^{*}(t)) \chi_{[0,T-\delta]}(t)\varphi_{H}
(t+\delta,s) \}ds \}dt] \nonumber \\
& \hspace{0.6cm} +\mathbb{E}[\int_{0}^{T}  p^{*}(t) \partial_{u}b^{*}(t)(u(t)-u^{*}(t)) dt].
\end{align}
where $(\tilde{p}^{*},\tilde{q}^{*})$ is an independant copy of $(p^{*},q^{*})$ defined on some probability space $(\tilde{\Omega},\tilde{\mathcal{F}},\tilde{\mathbb{P}})$.\\
\\
By substituting  $\left(\ref{PY}\right)$ in $\left(  \ref{J}\right)$ and using the terminal value of the BSDE, we get
\begin{align*}
0 & \leq \mathbb{E}[\int_{0}^{T} Y(t) \{ p^{*}(t) \partial_{x}b^{*}(t) + p^{*}(t+\delta) \partial_{\bar{x}}b^{*}(t+\delta) \chi_{[0,T-\delta]}(t) + \mathbb{\tilde{E}}[\tilde{p}^{*}(t)\partial_{m}\tilde{b}^{*}(t)(X^{*}({t}))]\\
& \hspace{0.4cm} +  \mathbb{\tilde{E}}[\tilde{p}^{*}(t+\delta)\partial_{\bar{m}}\tilde{b}^{*}(t+\delta)(X^{*}({t}))]\chi_{[0,T-\delta]}(t)- \alpha(t) + \int_{0}^{T} \{ \tilde{\mathbb{E}}[\tilde{q}^{*}(s)] \partial_{m}{{\sigma}}^{*}(t)(X^{*}(t))\varphi_{H}
(t,s)  \\
& \hspace{0.4cm} +  \tilde{\mathbb{E}}[\tilde{q}^{*}(s)]\partial_{\bar{m}}{{\sigma}}^{*}(t +\delta)(X^{*}(t))\chi_{[0,T-\delta]}(t)\varphi_{H}
(t+\delta,s) \}ds + \partial_{x}f^{*}(t)   \\
& \hspace{0.4cm} + \partial_{\bar{x}}f^{*}(t+\delta)\chi_{[0,T-\delta]}(t) + \mathbb{\tilde{E}}[\partial_{{m}}\tilde{f}^{*}(t)(X^{*}(t))] + \mathbb{\tilde{E}}[\partial_{\bar{m}}\tilde{f}^{*}(t+\delta)(X^{*}(t))]\chi_{[0,T-\delta]}(t) \} dt] \\ & \hspace{0.4cm} + \mathbb{E}[\int_{0}^{T} \{ p^{*}(t) \partial_{u}b^{*}(t) + \partial_{u}f^{*}(t)\} (u(t)-u^{*}(t))dt].
\end{align*}
Letting the integrand which contains $Y(t)$ equal to zero, we get 
\begin{align*}
\alpha(t) & = p^{*}(t) \partial_{x}b^{*}(t) + p^{*}(t+\delta) \partial_{\bar{x}}b^{*}(t+\delta) \chi_{[0,T-\delta]}(t) + \mathbb{\tilde{E}}[\tilde{p}^{*}(t)\partial_{m}\tilde{b}^{*}(t)(X^{*}({t}))] \\
&\hspace{0.4cm} +  \mathbb{\tilde{E}}[\tilde{p}^{*}(t+\delta)\partial_{\bar{m}}\tilde{b}^{*}(t+\delta)(X^{*}({t}))]\chi_{[0,T-\delta]}(t) + \int_{0}^{T}  \{ \tilde{\mathbb{E}}[\tilde{q}^{*}(s)]\partial_{m}{{\sigma}}^{*}(t)(X^{*}(t))\varphi_{H}
(t,s)  \\
& \hspace{0.4cm} + \tilde{\mathbb{E}}[\tilde{q}^{*}(s)]\partial_{\bar{m}}{{\sigma}}^{*}(t +\delta)(X^{*}(t))\chi_{[0,T-\delta]}(t)\varphi_{H}
(t+\delta,s) \}ds + \partial_{x}f^{*}(t) + \partial_{\bar{x}}f^{*}(t+\delta)\chi_{[0,T-\delta]}(t)  
 \\
& \hspace{0.4cm} + \mathbb{\tilde{E}}[\partial_{m}\tilde{f}^{*}(t)(X^{*}(t))] + \mathbb{\tilde{E}}[\partial_{\bar{m}}\tilde{f}^{*}(t+\delta)(X^{*}(t))]\chi_{[0,T-\delta]}(t).
\end{align*}
where, for simplicity of notations, we have set :
\\
\[%
\begin{array}
[c]{ll}%
\partial_{x}\tilde{b}^{*}(t) & := \partial_{x}b(t,\tilde{X}^{\ast}(t),\tilde{X}^{\ast}(t-\delta),M^{\ast}(t),M^{\ast}(t-\delta
),\tilde{u}^{\ast}(t)), \\
\partial_{m}\tilde{b}^{*}(t)(.) & := \partial_{m}b(t,\tilde{X}^{\ast}(t),\tilde{X}^{\ast}(t-\delta),M^{\ast}(t),M^{\ast}(t-\delta
),\tilde{u}^{\ast}(t))(.), \\
\partial_{m}\tilde{f}^{*}(t)(.) & := \partial_{m}f(t,\tilde{X}^{\ast}(t),\tilde{X}^{\ast}(t-\delta),M^{\ast}(t),M^{\ast}(t-\delta
),\tilde{u}^{\ast}(t))(.).
\end{array}
\]
\\
\\
We define now the Hamiltonian associated to our problem by:%
\[
H:\Omega\times\left[  0,T\right]  \times%
\mathbb{R}
\times
\mathbb{R}\times \mathcal{P}_{2}(\mathbb{R})\times\mathcal{P}_{2}(\mathbb{R})\times\mathcal{U}\times%
\mathbb{R}
\times%
\mathcal{R}
\rightarrow%
\mathbb{R}
\]
by%
\begin{equation}%
\begin{array}
[c]{ll}%
H(t,x,\overline{x},m,\overline{m},u,p,q(.)) & = f(t,x,\overline
{x},m,\overline{m},u)+p\text{\ }b(t,x,\overline{x},m,\overline{m},u) +%
\sigma(t,m,\overline{m}){\textstyle\int_{0}^{T}}
q(s)\varphi_{H}(s,t)ds.
\end{array}
\label{ham}%
\end{equation}
where $\mathcal{R}$ is the set of continuous functions from $[0,T]$ into $\mathbb{R}$. \\
\\
For $u\in\mathcal{A}_{\mathbb{F}}$ with corresponding solution $X=X^{u}$,
define, whenever solutions exist, $(p,q) := (p^{u},q^{u}) $, by the adjoint equation, in terms of the Hamiltonian, as follows:
\begin{equation}
\left\{
\begin{array}
[c]{lll}%
dp(t) & = &-\{ \partial_{x}H(t)+\mathbb{E}[\partial_{\bar{x}}H(t+\delta)\chi_{[0,T-\delta]}(t)|\mathcal{F}_{t}]\\

& + & \tilde{\mathbb{E}}[\partial_{m}\tilde{H}(t)(X(t))] + \mathbb{\tilde{E}}[\partial_{\bar{m}}\tilde{H}(t+\delta)(X(t)) \chi_{[0,T-\delta]}(t)]\}dt +q(t)dB^{H}%
(t);t\in\left[  0,T\right]  ,\\
\\
p(T) & = & \partial_{x}g(T)+ \mathbb{\tilde{E}}[\partial_{m}\tilde{g}(T)({X}(T))]
\end{array}
\right.  \label{eq:2}%
\end{equation}
We assume that $(p,q)$ is in $\mathcal{L}_{H}^{1,2}([0,T])\times \mathcal{L}_{H}^{1,2}([0,T]) $, the integral with respect to the fBm is understood in the fractional Wick-It\^{o}-Skorohod sense. \\

For simplicity of notations, we have put :
\[%
\begin{array}
[c]{ll}%
\partial_{x}H(t) & := \partial_{x}H(t,X(t),X(t-\delta),M(t),M(t-\delta),u(t),p(t),q(.)), \\
\partial_{m}\tilde{H}(t)(.) & := \partial_{m}H(t,\tilde{X}(t),\tilde{X}(t-\delta),M(t),M(t-\delta),\tilde{u}(t),\tilde{p}(t),\tilde{q}(.))(.),\\
\partial_{x}g(T) & :=\partial_{x}g(X(T),M(T)), \text{ \ } \partial_{m}\tilde{g}(T)(.) = \partial_{m}g(\tilde{X}(T),M(T))(
.).
\end{array}
\]
\begin{rem}\label{remarkimportant}
By the definition of the Hamiltonian given above, the time advanced BSDE (\ref{eq:2}) is expressed as a first part which appears to be linear and a second part with coefficients that contain the laws or more precisely the joint distribution of the solution processes. This type of Backward Stochastic Differential Equations was never been studied before. However, when there is no mean-field terms and no advance in time, several resolutions were proposed see for instance \cite{HOB}, \cite{HZ} and for a more general setting \cite{HP}. In the section devoted to the applications, we suggest some dynamics where the mean-field term  appears mainly in the terminal cost functional, the resolution of the BSDE (\ref{eq:2}) is in this case possible following a recursive procedure.
\end{rem}
We establish in the following theorem the necessary condition of optimality.
\begin{thm}[Necessary condition of optimality]We assume that the couple $(u^{\ast
},X^{\ast})$ is optimal. Suppose that there exists $p^{*}(t)$, $q^{*}(t)$,
 solutions of the adjoint equation $\left(\ref{eq:2}\right)$ associated to the pair $(u^{\ast},X^{\ast})$ and that $Y(t)q^{*}(t) \in \mathcal{L}_{H}^{1,2}([0,T])$ and $\psi^{*}_{\delta}(t)p^{*}(t) \in \mathcal{L}_{H}^{1,2}([0,T]) $. Then%
\[
0\leq\mathbb{E[}\int_{0}^{T}\partial_{u}H^{\ast}(t)\left(  u(t)-u^{\ast}(t)\right) dt].
\]
\end{thm}
\begin{proof}
Suppose that $u^{\ast}$ is an optimal control, then we have%
\begin{equation}%
\begin{array}
[c]{l}%
0\leq\text{ }\underset{\theta\rightarrow0}{\lim}\tfrac{1}{\theta}(J(u^{\theta
})-J(u^{\ast}))\\
=\mathbb{E[}p^{*}(T)Y(T)]+\mathbb{E[}%
{\textstyle\int_{0}^{T}}
\{\partial_{x}f^{*}(t)Y(t)+\partial_{\bar{x}}f^{*}(t)Y(t-\delta)\\
 \hspace{0.3cm}+ \mathbb{\tilde{E}}[\partial_{m}f^{*}(t)(\tilde{X}^{*}(t))\tilde{Y}(t)]+ \mathbb{\tilde{E}}[\partial_{\bar{m}}f^{*}(t)(\tilde{X}^{*}(t-\delta))\tilde{Y}(t-\delta)] +\partial_{u}f^{*}(t)\left( u(t)-u^{\ast}(t)\right)  \}dt].
\end{array}
\label{va}%
\end{equation}

Applying again the integration by parts formula of Proposition \ref{integ.part.form} to $p^{*}\left(  t\right)$ and $Y\left(  t\right)  $ then using Fubini's theorem, remplacing $\psi^{*}_{\delta}(t)$ by its value, and using a change of variable and the fact that $Y(t) = 0$ for all $t \in [-\delta,0]$, we get
\begin{align}\label{A1}
\mathbb{E}[A_{1}]  &  = \mathbb{E}[\int_{0}^{T}p^{*}(t)dY(t)]+ \mathbb{E}[\int_{0}^{T}Y(t)dp^{*}(t)]+ \mathbb{E}[\int_{0}^{T}\int_{0}^{T}
q^{*}(s) \psi^{*}_{\delta}(t) \varphi_{H}(t,s)ds dt] \nonumber \\
& = \mathbb{E}[\int_{0}^{T}Y(t) \{ p^{*}(t) \partial_{x}b^{*}(t) + p^{*}(t+\delta)\partial_{\bar{x}}b^{*}(t+\delta) \chi_{[0,T-\delta]}(t) +\mathbb{\tilde{E}}[\tilde{p}^{*}(t)\partial_{m}\tilde{b}^{*}(t)(X^{*}(t))] \nonumber\\
& \hspace{0.2cm} + \mathbb{\tilde{E}}[\tilde{p}^{*}(t+\delta)\partial_{\bar{m}}\tilde{b}^{*}(t+\delta)(X^{*}(t))]\chi_{[0,T-\delta]}(t) \nonumber  - \{ \partial_{x}H^{*}(t) +\mathbb{E}[\partial_{\bar{x}}H^{*}(t+\delta)\chi_{[0,T-\delta]}(t)|\mathcal{F}_{t}]\\
& \hspace*{0.2cm}  + \tilde{\mathbb{E}}[\partial_{m}\tilde{H}^{*}(t)(X^{*}(t))]+ \mathbb{E}[\mathbb{\tilde{E}}[\partial_{\bar{m}}\tilde{H}^{*}(t+\delta)(X^{*}(t)) \chi_{[0,T-\delta]}(t)]|\mathcal{F}_{t}] \} \nonumber  \\
& \hspace{0.2cm} + \int_{0}^{T}  \{ \tilde{\mathbb{E}}[\tilde{q}^{*}(s)] \partial_{m}{{\sigma}}^{*}(t)(X^{*}(t))\varphi_{H}
(t,s) + \tilde{\mathbb{E}}[\tilde{q}^{*}(s)]\partial_{\bar{m}}{{\sigma}}^{*}(t +\delta)(X^{*}(t)) \chi_{[0,T-\delta]}(t)\varphi_{H}
(t+\delta,s) \}ds \}dt] \nonumber\\
& \hspace*{0.2cm} +\mathbb{E}[\int_{0}^{T}  p^{*}(t) \partial_{u}b^{*}(t)(u(t)-u^{*}(t)) dt].]
\end{align}

where $A_{1} := \mathbb{E}[p^{*}(T)Y(T)]$, then by (\ref{va}), (\ref{A1}) and applying the definition of the Hamiltonian, we get the desired result, that is%
\[
0\leq\mathbb{E[}%
{\textstyle\int_{0}^{T}}
\partial_{u}H^{\ast}(t)\left(  u(t)-u^{\ast}(t)\right)  dt].
\]
\end{proof}
\section{Sufficient maximum principle}
In this section, we proove sufficient stochastic maximum principle.
\begin{thm}[Sufficient condition of optimality]\label{suff.cond} Let $u^{\ast}%
\in\mathcal{A}_{\mathbb{F}}$ with corresponding controlled state $X^{\ast}$. Suppose
that there exist $p^{*}(t)$, $q^{*}(t)$ solution
of the associated adjoint equation $\left(  \ref{eq:2}\right)$. Assume the following:
 
\begin{enumerate}
 \item $X^{u}(t)q^{*}(t) \in \mathcal{L}_{H}^{1,2}([0,T])$, $p^{*}(t)\sigma(t,M(t),M(t-\delta))\in \mathcal{L}_{H}^{1,2}([0,T])$ $\forall u \in \mathcal{A}_{\mathbb{F}}$.
\item (Concavity) The functions
\[%
\begin{array}
[c]{ll}%
(x,\bar{x},m,\bar{m},u) & \mapsto H(t,x,\bar{x},m,\bar{m},u,p^{*}(t),q^{*}(.)) ,\\
(x,m) & \mapsto g(x,m)\text{,}%
\end{array}
\]
are concave  for each $t \in [0,T]$ almost surely.
\item (Maximum condition)%
\begin{center}
$H(t, X^{*}(t),X^{*}(t-\delta), M^{*}(t),M^{*}(t-\delta),u^{*}(t), p^{*}(t),q^{*}(.)) 
= \hspace{2cm}\underset{u\in \mathcal{U}}{\text{ }\sup} H(t, X^{*}(t),X^{*}(t-\delta), M^{*}(t),M^{*}(t-\delta),u, p^{*}(t),q^{*}(.)) $
\label{maxQ}%
\end{center}
for all $t \in [0,T]$ almost surely.
\end{enumerate}

Then $(u^{\ast},X^{\ast})$ is an optimal couple for our problem.
\end{thm}
\begin{proof}
Let $u\in\mathcal{A}_{\mathbb{F}}$ be a generic admissible control, and $X = X^{u}$ the corresponding controlled state process. By the definition of the performance functional $J$ given by 
$\left(  \ref{P}\right)  $, we have 

\begin{equation}%
\begin{array}
[c]{lll}%
J(u)-J(u^{\ast}) & = & A_{2} + A_{3},
\end{array}
\label{J2}%
\end{equation}
where
\begin{align*}
&
\begin{array}
[c]{lll}%
A_{2} & := & \mathbb{E[}%
{\textstyle\int_{0}^{T}}
\{f(t)-f^{\ast}(t)\}dt],
\end{array}
\\
&
\begin{array}
[c]{lll}%
A_{3} & := & \mathbb{E[}g(T)-g^{\ast}(T)].
\end{array}
\end{align*}
Applying the definition of the Hamiltonian $\left(  \ref{ham}\right)  $, we
have%
\begin{equation}%
\begin{array}
[c]{ccc}%
A_{2} & = & \mathbb{E[}%
{\textstyle\int_{0}^{T}}
\{H(t)-H^{\ast}(t)-p^{*}(t){b}^{\prime}(t)-%
{\textstyle\int_{0}^{T}}
q^{*}(s){\sigma}^{\prime}(t)\varphi_{H}(s,t)ds\}dt]
\end{array}
\label{j1}%
\end{equation}
where we used the following notations
\begin{center}
\[%
\begin{array}
[c]{ll}%
b(t) & := b(t,X(t),X(t-\delta), M(t),M(t-\delta),u(t)), \\
b^{*}(t) & := b(t,X^{*}(t),X^{*}(t-\delta), M^{*}(t),M^{*}(t-\delta),u^{*}(t)), \\
f(t) &: = f(t,X(t),X(t-\delta), M(t),M(t-\delta),u(t)), \\
f^{*}(t) & := f(t,X^{*}(t),X^{*}(t-\delta), M^{*}(t),M^{*}(t-\delta),u^{*}(t)), \\
g(T) & := g(X(T), M(T)), \text{ \ }  g^{*}(T)  := g(X^{*}(T), M^{*}(T)), \\
\sigma(t) & := \sigma(t,M(t),M(t-\delta)), \text{ \ }\sigma^{*}(t) := \sigma(t, M^{*}(t), M^{*}(t-\delta)),\\
H(t) & := H(t,X(t),X(t-\delta),M(t),M(t-\delta),u(t),p^{*}(t),q^{*}(.)), \\
H^{*}(t) & := H(t,X^{*}(t),X^{*}(t-\delta),M^{*}(t),M^{*}(t-\delta),u^{*}(t),p^{*}(t),q^{*}(.)), \\
{b}^{\prime}(t) & :=  b(t)-b^{\ast}(t),\\
{\sigma}^{\prime}(t) & := \sigma(t)-\sigma^{\ast}(t)\\
{X}^{\prime}(t) & := X(t)- X^{\ast}(t) 
\end{array}
\]
\end{center}
Now using the concavity of $g$ and the terminal value of the BSDE (\ref{eq:2}) associated to  $(u^{*},X^{*})$, we get by Fubini's theorem
\begin{align*}
A_{3} & \leq \mathbb{E}[\partial_{x}g^{*}(T){X}^{\prime}(T)]  + \mathbb{E}[\tilde{\mathbb{E}}[\partial_{m}g^{*}(T)(\tilde{X}(T))\tilde{X}^{\prime}(T)]]\\
& = \mathbb{E}[(\partial_{x}g^{*}(T) + \tilde{\mathbb{E}}[\partial_{m}\tilde{g}^{*}(T)({X}(T))]){X}^{\prime}(T)] \\
& = \mathbb{E}[p^{*}(T){X}^{\prime}(T)] 
\end{align*}
Applying the integration by parts formula to $p^{*}(t)$ and $X^{\prime}(t)$, we get
\begin{align*}
\mathbb{E}[p^{*}(T)X^{\prime}(T)]& = \mathbb{E}[\int_{0}^{T}p^{*}(t)dX^{\prime}(t)] + \mathbb{E}[\int_{0}^{T}X^{\prime}(t) dp^{*}(t)] + 
\mathbb{E}[\int_{0}^{T}\int_{0}^{T}q^{*}(s)\sigma^{\prime}(t) \varphi_{H}(t,s)ds dt] \\
& = \mathbb{E}[\int_{0}^{T}p^{*}(t)b^{\prime}(t)dt]- \mathbb{E}[\int_{0}^{T} X^{\prime}(t)\{\partial_{x}{H}^{*}(t)  +
 \partial_{\bar{x}}H^{*}(t+\delta) \chi_{[0,T-\delta]}(t) \\
& \hspace*{0.5cm}  +  \mathbb{\tilde{E}}[\partial_{m}\tilde{H}^{*}(t)({X}^{*}(t))] +  \mathbb{\tilde{E}}[\partial_{\bar{m}}\tilde{H}^{*}(t)({X}^{*}(t))]\chi_{[0,T-\delta]}(t)\}dt]  \\
& \hspace*{0.5cm} + \mathbb{E}[\int_{0}^{T}\int_{0}^{T}q^{*}(s)\sigma^{\prime}(t) \varphi_{H}(t,s)dtds].
\end{align*}
Note that by the change of variables $r = t + \delta$, we have 
\begin{equation*}
\mathbb{E}[\int_{0}^{T-\delta} {X}^{\prime}(t) \partial_{\bar{x}}H^{*}(t+\delta)dt]  = \mathbb{E}[\int_{\delta}^{T} {X}^{\prime}(r-\delta) \partial_{\bar{x}}H^{*}(r)dr] = \mathbb{E}[\int_{0}^{T} {X}^{\prime}(r-\delta) \partial_{\bar{x}}H^{*}(r)dr] 
\end{equation*}
where we used that $\mathbb{E}[\int_{0}^{\delta} {X}^{\prime}(r-\delta) \partial_{\bar{x}}H^{*}(r)dr] = \mathbb{E}[\int_{- \delta}^{0} {X}^{\prime}(u) \partial_{\bar{x}}H^{*}(u+\delta)du] = 0$ since ${X}^{\prime}(u) = 0$ for all $u \in [-\delta,0]$, because $X^{*}(t) = X(t) = x_{0}(t)$ for all $t  \in [-\delta,0]$. \\
\\
Similarly, we get using the previous argument and by Fubini's theorem
\begin{equation*}
\mathbb{E}[\int_{0}^{T} X^{\prime}(t) \mathbb{\tilde{E}}[\partial_{\bar{m}} \tilde{H}^{*}(t+\delta)(X^{*}(t))]\chi_{[0,T-\delta]}(t)dt] = \mathbb{E}[\int_{0}^{T}\mathbb{\tilde{E}}[\partial_{\bar{m}}{H}^{*}(t)(\tilde{X}^{*}(t-\delta))\tilde{X}^{\prime}(t-\delta)]dt]
\end{equation*}
Hence, we have
\begin{align}\label{PXprim}
\mathbb{E}[p^{*}(T)X^{\prime}(T)] & = \mathbb{E}[\int_{0}^{T}p^{*}(t)b^{\prime}(t)dt]- \mathbb{E}[\int_{0}^{T} X^{\prime}(t)\partial_{x}H^{*}(t) dt] - 
\mathbb{E}[\int_{0}^{T} \partial_{\bar{x}}H^{*}(t) X^{\prime}(t-\delta) dt] \nonumber \\
& \hspace*{0.5cm}  - \mathbb{E}[\int_{0}^{T} \mathbb{\tilde{E}}[\partial_{m}H^{*}(t)(\tilde{X}^{*}(t))\tilde{X}^{\prime}(t)]dt]- \mathbb{E}[\int_{0}^{T} \mathbb{\tilde{E}}[\partial_{\bar{m}}H^{*}(t)(\tilde{X}^{*}(t-\delta))\tilde{X}^{\prime}(t-\delta)]dt] \nonumber \\
& \hspace*{0.5cm} + \mathbb{E}[\int_{0}^{T}\int_{0}^{T}q^{*}(s)\sigma^{\prime}(t) \varphi_{H}(t,s)dtds].
\end{align}
By (\ref{J2}), (\ref{j1}) and (\ref{PXprim}), we obtain
\begin{align*}
J(u)-J(u^{*}) & \leq \mathbb{E}[\int_{0}^{T} (H(t)-H^{*}(t))dt] - \mathbb{E}[\int_{0}^{T} \partial_{x}H^{*}(t) {X}^{\prime}(t)dt] - \mathbb{E}[\int_{0}^{T} \partial_{\bar{x}}H^{*}(t) {X}^{\prime}(t-\delta)dt]\\
& \hspace{0.8cm} - \mathbb{E}[\int_{0}^{T} \mathbb{\tilde{E}}[\partial_{m}H^{*}(t)(\tilde{X}^{*}(t))\tilde{X}^{\prime}(t)]dt]- \mathbb{E}[\int_{0}^{T} \mathbb{\tilde{E}}[\partial_{\bar{m}}H^{*}(t)(\tilde{X}^{*}(t-\delta))\tilde{X}^{\prime}(t-\delta)]dt]. \\
& \leq 0.
\end{align*}
due to the concavity assumption on $H$ and because $u^{*}$ satisfies the maximum condition, the first order derivative in $u^{*}$ vanishes.
\end{proof}

\section{Applications}
The main applications of mean-field dynamics that appear in the literature rely mainly on a dependence upon the probability measures through functions of scalar moments of the measures. More precisely, we assume that:
\begin{equation*}
b(t,x,\bar{x},m,\bar{m},u) = \hat{b}(t,x,\bar{x},(\psi_{1},m),(\psi_{2},\bar{m}),u),
\end{equation*}
\begin{equation*}
\sigma(t,m,\bar{m}) = \hat{\sigma}(t,(\phi_{1},m),(\phi_{2},\bar{m})),
\end{equation*}
\begin{equation*}
f(t,x,\bar{x},m,\bar{m},u) = \hat{f}(t,x,\bar{x},(\gamma_{1},m),(\gamma_{2},\bar{m}),u),
\end{equation*}
\begin{equation*}
g(x,m) = \hat{g}(x,(\gamma_{3},m)).
\end{equation*}
for some scalar functions $\psi_{1}$, $\psi_{2}$, $\phi_{1}$, $\phi_{2}$, $\gamma_{1}$, $\gamma_{2}$, $\gamma_{3}$ with at most quadratic growth at $\infty$. The functions $\hat{b}$, $\hat{f}$ are defined on $[0,T] \times \mathbb{R}\times \mathbb{R} \times \mathbb{R}\times \mathbb{R} \times \mathcal{U}$, the function $\hat{\sigma}$ is defined on $[0,T] \times \mathbb{R} \times \mathbb{R}$ and $\hat{g}$ is defined on $\mathbb{R} \times \mathbb{R}$. The notation $(\psi, m)$ denotes the integral of the function $\psi$ with respect to the probability measure $m$. The Hamiltonian that we defined in the previous section takes now the following form:
\begin{equation*}%
\begin{array}
[c]{ll}%
H(t,x,\overline{x},m,\overline{m},u,p,q(.)) & = \hat{f}(t,x,\overline
{x},(\gamma_{1},m),(\gamma_{2},\bar{m}),u)+p\text{\ }\hat{b}(t,x,\overline{x},(\psi_{1},m),(\psi_{2},\bar{m}),u)  \\
& \hspace{0.5cm}%
+ \hat{\sigma}(t,(\phi_{1},m),(\phi_{2},\bar{m})){\textstyle\int_{0}^{T}}
q(s)\varphi_{H}(s,t)ds.
\end{array}
\label{ham}%
\end{equation*}
According to the definition of the differentiability with respect to  functions of measures recalled in the preliminaries, the derivative of the Hamiltonian with respect to the measure $m$ for instance, is computed as follows :
\begin{align*}
\partial_{m}H(t,x,\overline{x},m,\overline{m},u,p,q(.))(x^{\prime}) & = \partial_{x^{\prime}} \hat{f} (t,x,\bar{x},(\gamma_{1},m),(\gamma_{2},\bar{m}),u) \gamma_{1}^{\prime}(x^{\prime})  \\
& \hspace{0.8cm}+ p \times \partial_{x^{\prime}} \hat{b} (t,x,\bar{x},(\psi_{1},m),(\psi_{2},\bar{m}),u) \psi_{1}^{\prime}(x^{\prime}) \\
& \hspace{0.8cm} + \partial_{x^{\prime}} \hat{\sigma} (t,(\phi_{1},m),(\phi_{2},\bar{m})) \phi_{1}^{\prime}(x^{\prime}){\textstyle\int_{0}^{T}}
q(s)\varphi_{H}(s,t)ds.
\end{align*}
The terminal value of the adjoint BSDE(\ref{eq:2}) which is $p(T)  = \partial_{x}g(T)+ \mathbb{\tilde{E}}[\partial_{m}\tilde{g}(T)({X}(T))]$, can be written in terms of the derivatives of the function $\hat{g}$ as follows: 
\begin{equation*}
p(T)  = \partial_{x}\hat{g}(X_{T},\mathbb{E}[\gamma_{3}(X_{T})])+ \mathbb{\tilde{E}}[\partial_{x^{\prime}}\hat{g}(\tilde{X}_{T},\mathbb{E}[\gamma_{3}(X_{T}))]\gamma_{3}^{\prime}(X_{T})
\end{equation*} 
where $\tilde{X}_{T}$ is an independant copy of $X_{T}$. We study in the following two applications that illustrate the previous results.
\subsection{Optimal consumption from a cash flow with delay}
We consider the problem of an optimal consumption  with a cash flow with delay $X := X^{\rho}$ given by 
\begin{equation}\label{wealth}
\left\{
\begin{array}
[c]{lll}%
dX(t) & =  [X(t-\delta)-\rho(t)]dt + \beta(t)dB^{H}(t) , \text{ \ } t \in [0,T],\\
X(t) & = x_{0}(t)  \text{ \ } t \in [-\delta,0]. 
\end{array}
\right.  %
\end{equation}
where $\rho$ is the relative consumption rate (our control), $x_{0}$ a bounded deterministic function, $\delta$ a strictly positive constant and $\beta$ is a given deterministic function in $\mathbb{L}_{H}^{2}([0,T])$. The integral with respect to the fBm is therefore a Wiener type integral.
 \\
\\
The problem we consider is to find the consumption rate $\rho^{*}$ such that
\begin{equation*}
J(\rho^{*}) = \sup_{\rho \in \mathcal{\bar{A}}_{\mathbb{F}}} J(\rho)
\end{equation*} 
where
\begin{equation*}
J(\rho) = \mathbb{E}[\int_{0}^{T} \log(\rho(t)) dt + \xi_{1} \mathbb{E}[X(T)]],
\end{equation*}
over the set $\mathcal{\bar{A}}_{\mathbb{F}}$ of admissible controls which are $\mathbb{F}$-adapted processes with values in $\mathbb{R}^{*}_{+}$, $\xi_{1} > 0$ is a given bounded $\mathcal{F}_{T}$-measurable random variable assumed to be in ${\mathcal{L}}_{H}^{1,2}([0,T])$, we also assume that $X$ exists and belongs to $L^{2}(\Omega \times [0,T])$.\\

Note that the running cost functional we consider in this example is the function $\rho(t) \mapsto \log(\rho(t))$ which is a utility function. Moreover, in order to control the fluctuations of the terminal time-value of the cash flow $X^{\rho}_{T}$, we chose to introduce its mean in the terminal cost functional. Therefore according to the notations used in the previous paragraph, the terminal cost functional is of mean-field type, more precisely it has the following form:
\begin{equation*}
g(X_{T}, \mathbb{P}_{X_{T}}) = g(\mathbb{P}_{X_{T}})= \xi_{1} \mathbb{E}[X_{T}]=\hat{g}_{\xi_{1}}((\textsc{Id},\mathbb{P}_{X_{T}}))
\end{equation*}
where $\hat{g}_{\xi_{1}}(x^{\prime})= \xi_{1} x^{\prime}$, therefore $\hat{g}^{\prime}_{\xi_{1}}(x^{\prime})= \xi_{1} $  a.s.\\
\\
The Hamiltonian of this control problem  is given by :
\begin{equation*}
H(t,x,\bar{x},\rho,p,q(.)) = \log(\rho) + (\bar{x} - \rho)p + \beta(t) \int_{0}^{T}q(s) \varphi_{H}(s,t)ds,
\end{equation*}
where $(p,q)$ is the solution of the associated adjoint BSDE 
\begin{equation}\label{bsde}
\left\{
\begin{array}
[c]{lll}%
dp(t) & =  - \mathbb{E}[p(t+\delta)\chi_{[0,T-\delta]}(t)|\mathcal{F}_{t}] dt + q(t) dB^{H}(t) , \text{ \ } t \in [0,T],\\
p(T) & =  \xi_{1}.
\end{array}
\right.  %
\end{equation}
Inspired by the resolution of the linear BSDE driven by a fractional Brownian motion with Hurst parameter $H > 1/2$ done in \cite{HOB} and \cite{HP}, we propose a resolution of the anticipated BSDE (\ref{bsde}) by solving  a sequence of linear BSDEs following this procedure :
\\
\\
\textbf{Step 1 } If $t \in [T-\delta,T]$, the previous BSDE takes the form 
\begin{equation*}
\left\{
\begin{array}
[c]{lll}%
dp(t) & =  q(t) dB^{H}(t) , \text{ \ } t \in  [T-\delta,T],\\
p(T) & =  \xi_{1}.
\end{array}
\right.  %
\end{equation*}
which has the solution 
\begin{equation*}
p(t) = \hat{\mathbb{E}}[\xi_{1} | \mathcal{F}_{t}],  \text{ \ } q(t) =  \hat{\mathbb{E}}[D^{H}_{t}\xi_{1} | \mathcal{F}_{t}], \text{ \ }t \in  [T-\delta,T]
\end{equation*}
where $\hat{\mathbb{E}}$ is the quasi-conditional expectation, see \cite{BHOZ} for its definition.\\
\\
\textbf{Step 2 :} If $t \in [T-2\delta, T-\delta]$ and $T- 2 \delta > 0$, we obtain the BSDE 
\begin{equation*}
\left\{
\begin{array}
[c]{lll}%
dp(t) & =  - \mathbb{E}[p(t+\delta)|\mathcal{F}_{t}] + q(t) dB^{H}(t) , \text{ \ } t \in  [T- 2\delta,T - \delta],\\
p(T- \delta)& \text{ \ } \text{known from step 1}.
\end{array}
\right.  %
\end{equation*}
this BSDE has an expicit solution given by 
\begin{equation*}
p(t) = \hat{\mathbb{E}}[p(T-\delta) + \int_{t}^{T} \theta(s) dt | \mathcal{F}_{t}], \text{ \ } q(t) = D^{H}_{t}p(t), \text{ \ } t \in  [T- 2\delta,T - \delta].
\end{equation*}
where $\theta(t) = \mathbb{E}[p(t+\delta)|\mathcal{F}_{t}]$ and $p(t+\delta)$ is known by step 1. \\
\\
We continue like this by induction up to and including step n, where n is such that $T - n\delta \leq 0 < T- (n-1)\delta$ and we solve the corresponding BSDE on the time interval $[0, T-(n-1)\delta]$ and we solve the corresponding BSDE on the time interval $[0, T-(n-1)\delta]$. \\
\\ 
Maximizing H with respect to $\rho$ gives the following first order condition for an optimal consumption rate $\rho^{*}$ :

\begin{equation*}
{\partial_{\rho} H^{*}(t)} = \frac{1}{\rho^{*}(t)} - p(t) = 0,
\end{equation*}
Then if \begin{equation}\label{condition}
 p(t) > 0  \text{ \ for all } t \in [0,T].
\end{equation}
 We get 
\begin{equation}\label{optimal-consumption}
\rho^{*}(t) = \frac{1}{p(t)} \text{ \ for all } t \in [0,T].
\end{equation}
where $p$ satisfies the previous anticipated BSDE.
\begin{thm}
Let $(p,q)$ be the solution of the BSDE (\ref{bsde}) and suppose that (\ref{condition}) holds. Then any optimal consumption rate $\rho^{*}$ satisfies (\ref{optimal-consumption}) and the corresponding optimal wealth $X^{*}$ is given by equation (\ref{wealth}).
\end{thm}
\subsection{Linear-Quadratique Problem with delay}
We consider now a Linear-Quadratique (LQ) model for a controlled process $X = X^{\alpha}$ given by the following delayed stochastic differential equation: 
\begin{equation}\label{wealth1}
\left\{
\begin{array}
[c]{lll}%
dX(t) & =  [\beta_{1}(t)X(t-\delta)+ \alpha(t)]dt + \beta_{2}(t)dB^{H}(t) , \text{ \ } t \in [0,T],\\
X(t) & = x_{0}(t)  \text{ \ } t \in [-\delta,0]. 
\end{array}
\right.  %
\end{equation}

where $\delta > 0$ is a given constant, $\beta_{1}$, $x_{0}$ are given bounded deterministic functions, $\beta_{2}$ is a given deterministic function in $\mathbb{L}_{H}^{2}([0,T])$. The integral with respect to the fBm is therefore a Wiener type integral and $\alpha \in \mathcal{A}_{\mathbb{F}}$ is our control process, the set $\mathcal{A}_{\mathbb{F}}$ are the admissible controls assumed to be square integrable $\mathbb{F}$-adapted processes with real values.
\\
\\
We want to minimize the expected value of $(X_{T}-\mathbb{E}[X_{T}])^{2}$ which is the variance of $X_{T}$ with a minimal average use of energy, measured by the integral $\mathbb{E}[\int_{0}^{T}\alpha^{2}(t)dt]$, more precisely, the performance functional is of the following form:
\begin{equation}\label{LQJ}
J(\alpha) = - \frac{1}{2} (\textsc{Var}(X_{T}) + \mathbb{E}[ \int_{0}^{T}\alpha^{2}(t)dt]).
\end{equation}
Our goal is therefore to find the control process $\alpha^{*} \in \mathcal{A}_{\mathbb{F}}$, such that 
\begin{equation}\label{LQ}
J(\alpha^{*}) = \sup_{\alpha \in \mathcal{A}_{\mathbb{F}}} J(\alpha)
\end{equation}
\begin{rem}
Including the variance of the state process in the cost functional in order to keep it small is a way to control its sensitivity with respect to the possible variations of the random events. The form of this cost functional is inspired from \cite{JY}.
\end{rem}

Note that the terminal cost functional of our problem has the following form:
\begin{equation*}
g(X_{T}, \mathbb{P}_{X_{T}}) = \hat{g}(X_{T},(\textsc{Id},\mathbb{P}_{X_{T}})) = - \frac{1}{2}(X_{T}-\mathbb{E}[X_{T}])^{2}
\end{equation*}
where $\hat{g}(x,x^{\prime}) = -\frac{1}{2}(x-x^{\prime})^{2}$, therefore $\partial_{x}\hat{g}(x,x^{\prime})  = -(x-x^{\prime}) =  - \partial_{x^{\prime}}\hat{g}(x,x^{\prime})$. Therefore the terminal value of the solution of the adjoint BSDE is:
\begin{align*}
p(T)  & = \partial_{x}\hat{g}(X_{T},(\textsc{Id},\mathbb{P}_{X_{T}}))+ \mathbb{\tilde{E}}[\partial_{x^{\prime}}\hat{g}(\tilde{X}_{T},(\textsc{Id},\mathbb{P}_{X_{T}}))] \\
& = -(X_{T}- \mathbb{E}[X_{T}]) + \tilde{\mathbb{E}}[\tilde{X}_{T}- \mathbb{E}[X_{T}]] \\
& = -(X_{T}- \mathbb{E}[X_{T}])
\end{align*}
where we used the fact that $\tilde{X}$ and $X$ have the same distribution. 
\\
\\
The Hamiltonian of our control problem takes the following form:
\begin{equation*}
H(t,x,\bar{x},\alpha,p,q(.)) = -\frac{1}{2} \alpha^{2} + (\beta_{1}(t)\bar{x} + \alpha)p + \beta_{2}(t) \int_{0}^{T}q(s) \varphi_{H}(s,t)ds,
\end{equation*}
where $(p,q)$ is the solution of the associated adjoint BSDE: 
\begin{equation*}
\left\{
\begin{array}
[c]{lll}%
dp(t) & =  - \beta_{1}(t+\delta)\mathbb{E}[p(t+\delta)\chi_{[0,T-\delta]}(t)|\mathcal{F}_{t}] dt + q(t) dB^{H}(t) , \text{ \ } t \in [0,T],\\
p(T) & = \xi_{2}(T).
\end{array}
\right.  %
\end{equation*}
where we put $\xi_{2}(T)= -(X_{T}-\mathbb{E}[X_{T}])$, we proceed now as we did in the previous exemple by solving a sequence of linear BSDEs.
\\
\\
The function $\alpha \mapsto H(t,X(t),X(t-\delta),\alpha(t),p(t),q(.))$ is maximal when
\begin{equation}\label{optimal}
\alpha(t) = \alpha^{*}(t) = p^{*}(t)
\end{equation}
where $p^{*}$ satisfies: 
\begin{equation}\label{Boptimal}
\left\{
\begin{array}
[c]{lll}%
dp^{*}(t) & =  - \beta_{1}(t+\delta)\mathbb{E}[p^{*}(t+\delta)\chi_{[0,T-\delta]}(t)|\mathcal{F}_{t}] dt + q^{*}(t) dB^{H}(t) , \text{ \ } t \in [0,T],\\
p(T) & = \xi_{2}^{*}(T).
\end{array}
\right.  %
\end{equation}
where $\xi_{2}^{*}(T)= -(X_{T}^{*}-\mathbb{E}[X_{T}^{*}])$. Therefore, we have proved the following theorem. 
\begin{thm}
The optimal control $\alpha^{*}$ of the LQ problem (\ref{LQ}) is given by (\ref{optimal}), where $(X^{*},p^{*},q^{*})$ solve the couple of systems (\ref{wealth1}) and (\ref{Boptimal}) of forward-backward stochastic differential equations.
\end{thm}
\section*{Acknowledgement}

The first author is supported by the Erasmus + International Credit mobility between Linnaeus University and Cadi Ayyad University for the academic year 2016-2017. The research of the third author is carried out with the support of the Norwegian Research
Council, within the research project Challenges in Stochastic Control,
Information and Applications (STOCONINF), project number 250768/F20.

\end{document}